\documentclass[sn-mathphys-num]{sn-jnl}

\usepackage{graphicx}%
\usepackage{multirow}%
\usepackage{amsmath,amssymb,amsfonts}%
\usepackage{amsthm}%
\usepackage{mathrsfs}%
\usepackage[title]{appendix}%
\usepackage{xcolor}%
\usepackage{textcomp}%
\usepackage{manyfoot}%
\usepackage{booktabs}%
\usepackage{algorithm}%
\usepackage{algorithmicx}%
\usepackage{algpseudocode}%
\usepackage{listings}%
\usepackage{cases}%

\theoremstyle{thmstyleone}%
\newtheorem{theorem}{Theorem}
\theoremstyle{thmstyletwo}%
\newtheorem{remark}{Remark}%
\newtheorem{assumption}{Assumption}%

\theoremstyle{thmstylethree}%
\newtheorem{lemma}{Lemma}
\begin{document}

\title[Article Title]{Bilateral boundary finite-time stabilization of 2 × 2 linear first-order hyperbolic systems with spatially varying coefficients}


\author[1,2]{\fnm{Wei} \sur{Sun}}\email{058056@chu.edu.cn}

\author*[1]{\fnm{Jing} \sur{Li}}\email{lijingmath@xidian.edu.cn}
\equalcont{These authors contributed equally to this work.}

\author[2]{\fnm{Liangyu} \sur{Xu}}\email{058059@chu.edu.cn}
\equalcont{These authors contributed equally to this work.}

\affil*[1]{\orgdiv{School of Mathematics and Statistics}, \orgname{Xidian University}, \orgaddress{\street{Xifeng Road}, \city{Xi'an}, \postcode{710126}, \state{Shaanxi}, \country{People’s Republic of China}}}

\affil[2]{\orgdiv{School of Mathematics and Big Data}, \orgname{Chaohu University}, \orgaddress{\street{Bantang Road}, \city{Hefei}, \postcode{238000}, \state{Anhui}, \country{People’s Republic of China}}}


\abstract{This paper presents bilateral control laws for one-dimensional(1-D) linear 2 × 2 hyperbolic first-order systems (with spatially varying coefficients). Bilateral control means there are two actuators at each end of the domain. This situation becomes more complex as the transport velocities are no longer constant, and this extension is nontrivial. By selecting the appropriate backstepping transformation and target system, the infinite-dimensional backstepping method is extended and a full-state feedback control law is given that ensures the closed-loop system converges to its zero equilibrium in finite time. The design of bilateral controllers enables a potential for fault-tolerant designs.}




\maketitle

\section{Introduction}\label{3:sec1}
The stability and stabilization of hyperbolic systems are primarily addressed through three methods: the characteristic method\cite{bib3.25,bib3.26,bib3.27}, the control Lyapunov function method\cite{bib3.28,bib3.29,bib3.30,bib3.31,bib3.32,bib3.33,bib3.34} and the backstepping method. Among them, the backstepping method has been shown to be a widely used and efficient approach for the control of partial differential equations (PDEs). It was first utilized to develop feedback control laws and observers for 1-D reaction-diffusion PDEs\cite{bib3.1} and has later been extended to various other systems\cite{bib3.2,bib3.3,bib3.4}. One of the representative applications involves the extension of the stabilization problem to 1-D hyperbolic systems\cite{bib3.5,bib3.6,bib3.7,bib3.24} with the introduction of the backstepping method, which makes the explicit design of controllers for 1-D hyperbolic systems possible.

Boundary stabilization of linear 1-D hyperbolic systems is an important topic in control theory, involving multiple fields such as mathematics, physics and engineering. The study of hyperbolic systems is crucial for understanding and controlling various physical processes, such as transmission line\cite{bib3.8}, traffic flow \cite{bib3.9,bib3.10}, heat exchanger\cite{bib3.11,bib3.23} and open channel flow\cite{bib3.12,bib3.13}. The current mainstream research focuses primarily on 1-D hyperbolic systems. Based on the number of actuators, they are mainly divided into unilateral control and bilateral control. 

Unilateral control, as the name suggests, refers to applying actuators at only one side of the boundary. For example, Hu et al. \cite{bib3.14,bib3.15} utilize the Volterra transformation to design an appropriate form for the target system and employ a classical backstepping controller to achieve the finite time convergence of the zero equilibrium for linear hyperbolic systems. Later, a new target system, distinct from the one in \cite{bib3.14}, is presented in \cite{bib3.16}, that achieves the minimum-time convergence to its zero equilibrium. In \cite{bib3.17}, the authors introduce a straightforward and novel proof regarding the optimal finite control time for linear coupled hyperbolic systems with spatially-varying coefficients.

Bilateral control refers to applying actuators on both sides of the boundary. The main advantages of bilateral control lie on 1) although the number of controllers increases, the total amount of controlled quantity may decrease in certain situations\cite{bib3.18}; 2) the control effect will be better, i.e, the convergence time will be shorter\cite{bib3.19}; 3) the robustness of the system will be enhanced, as a bilateral design can be made fault-tolerant. Once a fault is identified, it is sufficient to switch to unilateral control, which means using only the remaining controller. Of course, this design requires additional tools, such as extra observers for fault detection, which may require at least one sensor. From this perspective, this scheme provides a potential for the fault-tolerant control design of the linear hyperbolic systems with spatially varying coefficients. In \cite{bib3.20}, Fredholm transform is used to design a new target system, and then provide a bilateral control scheme on interval $[0,1]$ (i.e., applying actuators at boundaries $w=0$ and $w=1$), which makes a 1-D linear hyperbolic system (with constant transmission speed) converge to zero in minimal time. Subsequently, this result is extended to general heterodirectional hyperbolic systems\cite{bib3.19}. Vazquez and Krstic \cite{bib3.18} study the finite-time convergence problem of 2-state heterodirectional linear hyperbolic systems with equal constant transport velocities and solve the stability issue of bilateral control on interval $[-1,1]$ (i.e.,applying actuators at boundaries $w=-1$ and $w=1$). This method is found not applicable to the cases with unequal transport velocities, mainly because the well-posedness of the kernel equation cannot be guaranteed. Therefore, this method requires to be refined. So mainly built on \cite{bib3.18}, a more general scenario is directly investigated in this paper, where the transport velocity is a function of spatial variables. Additionally, our method will be valid for any constant transport velocities. This expansion is also not trivial because the original target system is no longer applicable, and thus the technical difficulty of this paper mainly lies on designing an appropriate target system to ensure the well-posedness of the kernel equation, while maintaining the excellent performance of the target system.

The structure of the paper is outlined below. Following this introduction, the subsequent section will delve into the model equations and associated notations. Section \ref{3:sec3} will focus on formulating a boundary feedback control law for the hyperbolic system via the backstepping transformation technique. In Section \ref{3:sec4}, a simulation result is given. Lastly, a concise summary of the entire article will be presented in Section \ref{3:sec5}.

\section{System Description and Control Objective}\label{3:sec2}
In this paper, the following $2\times 2$ linear hyperbolic PDE system is primarily considered
\begin{numcases}{}
\frac{\partial u}{\partial t}(w,t)+\lambda(w)\frac{\partial u}{\partial y}(w,t)=a(w)u(w,t)+b(w)v(w,t)\label{3:eq1}\\
\frac{\partial v}{\partial t}(w,t)-\mu(w)\frac{\partial v}{\partial y}(w,t)=c(w)u(w,t)+d(w)v(w,t)\label{3:eq2}\\
u(-1,t)=U_{1}(t),~~~~~v(1,t)=U_{2}(t)\label{3:eq3}
\end{numcases}
where $(w,t)\in\{(w,t)\in \mathbb{R}^{2}|w\in[-1,1], t\geq0\}$, $\lambda(\cdot),\mu(\cdot)\in C^{1}([-1,1]) >0$, coupling coefficients $a(\cdot), b(\cdot),c(\cdot),d(\cdot) \in C([-1,1])$, and $u$ and $v$ are the system states. 

The initial condition of system (\ref{3:eq1})-(\ref{3:eq3}) is 
\begin{align*}
u(w,0)=u_{0}(w),\quad v(w,0)=v_{0}(w),
\end{align*}
where $u_{0}$ and $v_{0}$ in $L^{2}([-1,1])$.
Without loss of generality, we also make the following assumption $a(w)=d(w)=0, \forall w\in[-1,1]$ (such coupling terms can be removed by using a similar coordinate
transformation as presented in \cite{bib3.6} and \cite{bib3.21}).
\begin{assumption}\label{3:assum1}
Assume that $\lambda(w)$ and $\mu(w)$ satisfy one of the following three cases
\begin{numcases}{}
\lambda(w)=\mu(-w), \quad\quad\forall w\in [-1,1]\label{3:eq4}\\
\lambda(w)> \mu(-w), \quad\quad\forall w\in [-1,1] \label{3:eq5}\\
\lambda(w)< \mu(-w), \quad\quad\forall w\in [-1,1].\label{3:eq6}
\end{numcases}
\end{assumption}

\begin{remark}\label{3:rem1}
The discussed bilateral control here and the range of variable $w\in[-1,1]$ will result in the kernel equations becoming more complex. To ensure the uniqueness of the kernel functions, this assumption is necessary.
\end{remark}

\begin{remark}\label{3:rem2}
If $\lambda$ and $\mu$ are constant real numbers, then our conclusion holds for any $\lambda,\mu >0$.
\end{remark}

The primary objective of this paper is to devise feedback control inputs $U_{1}(t)$ and $U_{2}(t)$ to facilitate system (\ref{3:eq1})-(\ref{3:eq3}) reach the zero equilibrium in finite time.

\section{Controller Design}\label{3:sec3}
In this section, a boundary feedback law will be estalished for the hyperbolic system (\ref{3:eq1})-(\ref{3:eq3}). By adopting the backstepping methodology, system (\ref{3:eq1})-(\ref{3:eq3}) will be mapped into a target system that exhibits favorable stability characteristics. Next, we will proceed with the controller design for the three cases (\ref{3:eq4})-(\ref{3:eq6}) respectively.
\subsection{Case 1: $\lambda(w)=\mu(-w)$}\label{3:subsec3.1}
Define the following target system
\begin{numcases}{}
\frac{\partial \alpha}{\partial t}(w,t)+\lambda(w)\frac{\partial \alpha}{\partial w}(w,t)=0\label{3:eq7}\\
\frac{\partial \beta}{\partial t}(w,t)-\mu(w)\frac{\partial \beta}{\partial w}(w,t)=0\label{3:eq8}\\
\alpha(-1,t)=\beta(1,t)=0,\label{3:eq9}
\end{numcases}
where initial condition $\alpha(w,0)=\alpha_{0}(w)$ and $\beta(w,0)=\beta_{0}(w)$ in $L^{2}([-1,1])$. 

To demonstrate the finite-time convergence to the zero equilibrium, the explicit solution of (\ref{3:eq7}) can be derived as follows. First, define
\begin{align}
\phi_{1}(w)=\int_{0}^{w}\frac{1}{\lambda(y)}\mathrm{d}y, \quad w\in[-1,1],\label{3:eq10}\\
\phi_{2}(w)=\int_{0}^{w}\frac{1}{\mu(y)}\mathrm{d}y, \quad w\in[-1,1].\label{3:eq11}
\end{align}

Notice that both $\phi_{1}$ and $\phi_{2}$ are monotonically increasing in terms of $w\in[-1,1]$, hence $\phi_{1}$ and $\phi_{2}$ are invertible. The lemma below illustrates that the target system (\ref{3:eq7})-(\ref{3:eq9}) will converge to the zero equilibrium in finite time.

\begin{lemma}\label{3:lem1}
The system (\ref{3:eq7})-(\ref{3:eq9}) reaches its zero equilibrium in finite time.
\end{lemma}
\begin{proof}
Using the method of characteristics, we can solve the explicit solution of (\ref{3:eq7}) and (\ref{3:eq9}) as
\begin{align*}
\alpha(w,t)=
\begin{cases}
\alpha_{0}(\phi_{1}^{-1}(\phi_{1}(w)-t))  \quad t<\phi_{1}(w)-\phi_{1}(-1)\\
0 \quad\quad\quad\quad\quad\quad\quad\quad~~ t\geq\phi_{1}(w)-\phi_{1}(-1)
\end{cases}
\end{align*}
and
\begin{align*}
\beta(w,t)=
\begin{cases}
\beta_{0}(\phi_{2}^{-1}(\phi_{2}(w)+t))  \quad t<\phi_{2}(1)-\phi_{2}(w)\\
0 \quad\quad\quad\quad\quad\quad\quad\quad~~ t\geq\phi_{2}(1)-\phi_{2}(w).
\end{cases}
\end{align*}
Thus, when $t\geq \mathrm{max}\{\phi_{1}(1)-\phi_{1}(-1), \phi_{2}(1)-\phi_{2}(-1)\}$, where $\phi_{1}(1)-\phi_{1}(-1)=\phi_{2}(1)-\phi_{2}(-1)$, we have $\alpha\equiv\beta\equiv0$.
\end{proof}
\begin{remark}\label{3:rem3}
The zero equilibrium of (\ref{3:eq7})-(\ref{3:eq9}) with any initial conditions $(\alpha_{0},\beta_{0})\in L^{2}([-1,1])$ is exponentially stable in the $L^{2}$ sense. This conclusion is classical, for a similar proof, refer to \cite{bib3.6}. 
\end{remark}
In order to map the original system (\ref{3:eq1})-(\ref{3:eq3}) to target system (\ref{3:eq7})-(\ref{3:eq9}), the following backstepping(Volterra) transformation is considered
\begin{align}
\alpha(w,t)&=-\int_{-w}^{w}u(z,t)L^{11}(z,w)+v(z,t)L^{12}(z,w)\mathrm{d}z+u(w,t),\label{3:eq12}\\
\beta(w,t)&=-\int_{-w}^{w}u(z,t)L^{21}(z,w)+v(z,t)L^{22}(z,w)\mathrm{d}z+v(w,t).\label{3:eq13}
\end{align}

\begin{remark}\label{3:rem4}
Given that the inverse backstepping transformation requires solving a second-kind Volterra integral equation, inverse transform of transformation (\ref{3:eq12})-(\ref{3:eq13}) is invariably feasible under relatively lenient conditions \cite{bib3.22} (such as when the kernel is continuous or even bounded, which it is in our scenario). 
\end{remark}

By differentiating (\ref{3:eq12}) and (\ref{3:eq13}) with respect to $w$ and $t$ and substituting them back into the target system (\ref{3:eq7})-(\ref{3:eq8}), two uncoupled $2\times2$ systems of hyperbolic 1-D equations can be obtained, that the kernels must satisfy.

First, kernels $L^{11}$ and $L^{12}$ safisfy 
\begin{align}
\lambda(w)L_{w}^{11}(z,w)+\lambda(z)L_{z}^{11}(z,w)=-\lambda^{\prime}(z)L^{11}(z,w)-c(z)L^{12}(z,w), \label{3:eq14}\\
\lambda(w)L_{w}^{12}(z,w)-\mu(z)L_{z}^{12}(z,w)=-b(z)L^{11}(z,w)+\mu^{\prime}(z)L^{12}(z,w), \label{3:eq15}
\end{align}
with boundary conditions
\begin{align}
L^{11}(-w,w)=0,\quad\quad L^{12}(w,w)=\frac{b(w)}{\lambda(w)+\mu(w)}\doteq h_{1}(w). \label{3:eq16}
\end{align}

Second, kernels $L^{21}$ and $L^{22}$ safisfy 
\begin{align}
\mu(w)L_{w}^{22}(z,w)+\mu(z)L_{z}^{22}(z,w)=b(z)L^{21}(z,w)-\mu^{\prime}(z)L^{22}(z,w), \label{3:eq17}\\
\mu(w)L_{w}^{21}(z,w)-\lambda(z)L_{z}^{21}(z,w)=\lambda^{\prime}(z)L^{21}(z,w)+c(z)L^{22}(z,w), \label{3:eq18}
\end{align}
with boundary conditions
\begin{align}
L^{22}(-w,w)=0,\quad\quad L^{21}(w,w)=-\frac{c(w)}{\lambda(w)+\mu(w)}\doteq h_{2}(w). \label{3:eq19}
\end{align}
Both equations (\ref{3:eq14})-(\ref{3:eq16}) and (\ref{3:eq17})-(\ref{3:eq19}) evolve separately in domain $E=\{(z,w)\in \mathbb{R}^{2}|-1\leq-w\leq z\leq w\leq 1\}$ (see Fig. \ref{3:fig1}).

The following theorem illustrates the well-posedness of kernel $L$.
\begin{theorem}\label{3:them1}
Hyperbolic PDE (\ref{3:eq14})-(\ref{3:eq19}) exists a unique solution $L^{11},L^{12},L^{21},L^{22} \in C(E)$.
\end{theorem}
\begin{proof}
Noting that equations (\ref{3:eq14})-(\ref{3:eq16}) and (\ref{3:eq17})-(\ref{3:eq19}) share the same structure, we will only present the proof of the well-posedness for kernels $L^{11}$ and $L^{12}$ here. The well-posedness for kernels $L^{21}$ and $L^{22}$ can be derived similarly.

Define $E=E_{1}\cup E_{2}=\{(z,w)\in \mathbb{R}^{2}|-w\leq z\leq w, 0\leq w\leq 1\}\cup \{(z,w)\in \mathbb{R}^{2}|w\leq z\leq -w, -1\leq w\leq 0\}$. The kernel equations evolve on $E_{1}$ and $E_{2}$, respectively. It can be clearly seen that if the kernel equations are well-posedness in domain $E_{1}$, it must also hold in $E_{2}$  due to a symmetry argument. Specifically, switching the variables from $(z,w)$ to $(\tilde{z},\tilde{w})=(z,-w)$ can transform domain $E_{2}$ into $E_{1}$ and preserve the structure of the equations, except for a change in symbol.

To transform PDEs (\ref{3:eq14})-(\ref{3:eq15}) into integral equations by using the method of characteristics, it is necessary to define
\begin{align}
\phi_{3}(w)=\phi_{1}(w)+\phi_{2}(w), \label{3:eq20}
\end{align}
where $\phi_{1}$ and $\phi_{2}$ are defined by (\ref{3:eq10}) and (\ref{3:eq11}). Note that $\phi_{3}$ is also monotonically increasing and invertible, due to the definition of $\phi_{1}$ and $\phi_{2}$. Furthermore, one can see that $\phi_{i},\phi_{i}^{-1}\in C^{1}([-1,1]),i=1,2,3.$
\begin{figure}
  \centering
  \includegraphics[width=7.5cm,height=7.0cm]{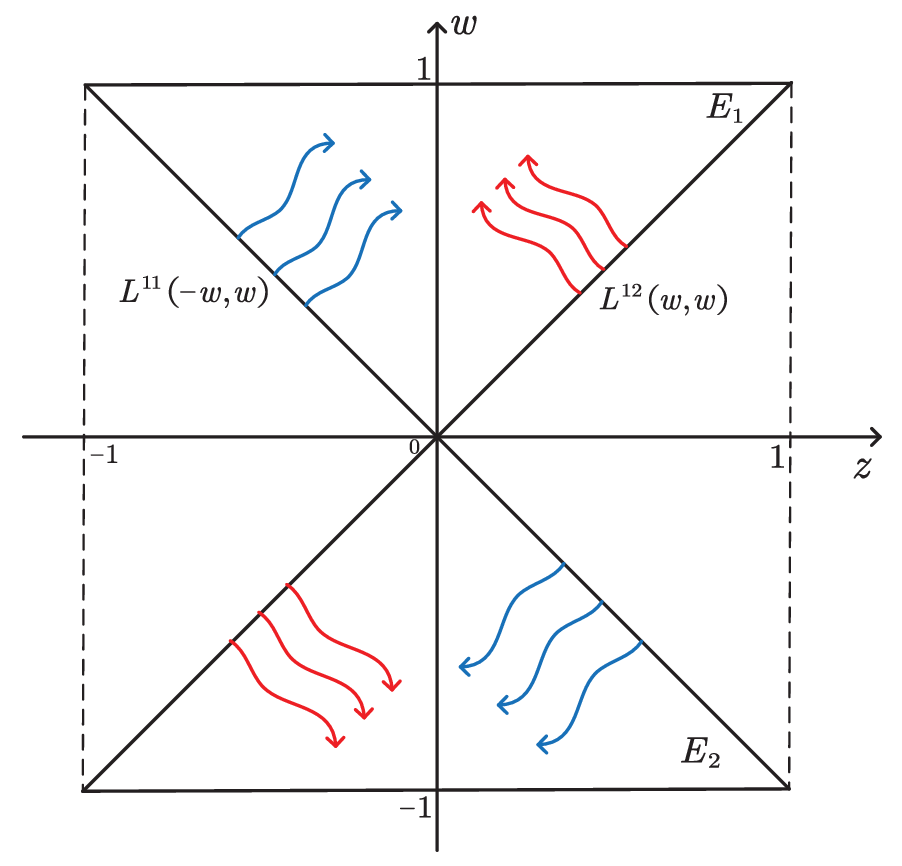}
  \caption{Characteristics lines for the kernel equations. Blue: characteristic lines for $L^{11}$. Red: characteristic lines for $L^{12}$.}\label{3:fig1}
\end{figure}

Characteristics of the $L^{11}$ kernel: $\forall (z,w)\in E_{1}$, define the following characteristic lines $(z_{1}(z,w,\cdot),w_{1}(z,w,\cdot))$
\begin{align}
&\begin{cases}\label{3:eq21}
\frac{\mathrm{d}w_{1}}{\mathrm{d}t}(z,w,t)=\lambda(w_{1}(z,w,t)),\quad t\in[0,t_{1}^{F}(z,w)],\\
w_{1}(z,w,0)=w_{1}^{0}(z,w),\quad w_{1}(z,w,t_{1}^{F}(z,w))=w,
\end{cases}\\
&\begin{cases}\label{3:eq22}
\frac{\mathrm{d}z_{1}}{\mathrm{d}t}(z,w,t)=\lambda(z_{1}(z,w,t)),\quad t\in[0,t_{1}^{F}(z,w)],\\
z_{1}(z,w,0)=z_{1}^{0}(z,w)=-w_{1}^{0}(z,w),\quad z_{1}(z,w,t_{1}^{F}(z,w))=z.
\end{cases}
\end{align}
These lines (see Fig. \ref{3:fig1}) start at coordinate $(-w_{1}^{0}(z,w),w_{1}^{0}(z,w))$ and end at position $(z,w)$. Further, we can calculate that
\begin{align}
\begin{cases}\label{3:eq23}
w_{1}(z,w,t)=\phi_{1}^{-1}(t+\frac{\phi_{1}(w)-\phi_{1}(z)}{2})\\
z_{1}(z,w,t)=\phi_{1}^{-1}(t+\frac{\phi_{1}(z)-\phi_{1}(w)}{2})
\end{cases}
\end{align}
where $t\in[0,t_{1}^{F}(z,w)]$ with
\begin{align}\label{3:eq24}
t_{1}^{F}(z,w)=\frac{\phi_{1}(z)+\phi_{1}(w)}{2}.
\end{align} 

Characteristics of the $L^{12}$ kernel: $\forall (z,w)\in E_{1}$, define the following characteristic lines $(z_{2}(z,w,\cdot),w_{2}(z,w,\cdot))$
\begin{align}
&\begin{cases}\label{3:eq25}
\frac{\mathrm{d}w_{2}}{\mathrm{d}t}(z,w,t)=\lambda(w_{2}(z,w,t)),\quad t\in[0,t_{2}^{F}(z,w)],\\
w_{2}(z,w,0)=w_{2}^{0}(z,w),\quad w_{2}(z,w,t_{2}^{F}(z,w))=w,
\end{cases}\\
&\begin{cases}\label{3:eq26}
\frac{\mathrm{d}z_{2}}{\mathrm{d}t}(z,w,t)=-\mu(z_{2}(z,w,t)),\quad t\in[0,t_{2}^{F}(z,w)],\\
z_{2}(z,w,0)=z_{2}^{0}(z,w)=w_{2}^{0}(z,w),\quad z_{2}(z,w,t_{2}^{F}(z,w))=z.
\end{cases}
\end{align}
These lines (see Fig. \ref{3:fig1}) start at coordinate $(w_{2}^{0}(z,w),w_{2}^{0}(z,w))$ and end at position $(z,w)$. Further, we also can calculate that
\begin{align}
\begin{cases}\label{3:eq27}
w_{2}(z,w,t)=\phi_{1}^{-1}(t+\phi_{1}(\phi_{3}^{-1}(\phi_{1}(w)+\phi_{2}(z))))\\
z_{2}(z,w,t)=\phi_{2}^{-1}(-t+\phi_{2}(\phi_{3}^{-1}(\phi_{1}(w)+\phi_{2}(z))))
\end{cases}
\end{align}
where $t\in[0,t_{2}^{F}(z,w)]$ with
\begin{align}\label{3:eq28}
t_{2}^{F}(z,w)=\phi_{1}(w)-\phi_{1}(\phi_{3}^{-1}(\phi_{1}(w)+\phi_{2}(z))).
\end{align}

By integrating the kernel equations (\ref{3:eq14}) and (\ref{3:eq15}) along with their respective characteristic lines and utilizing the boundary conditions (\ref{3:eq16}), we obtain
\begin{align*}
L^{11}(z,w)=\int_{0}^{t_{1}^{F}(z,w)}&[-\lambda^{\prime}(z_{1}(z,w,t))L^{11}(z_{1}(z,w,t),w_{1}(z,w,t))\\
&-c(z_{1}(z,w,t))L^{12}(z_{1}(z,w,t),w_{1}(z,w,t))]\mathrm{d}t
\end{align*}
and
\begin{align*}
L^{12}(z,w)=h_{1}(\phi^{-1}_{3}(\phi_{1}(w)+\phi_{2}(z)))+&\int_{0}^{t_{2}^{F}(z,w)}\!\!\!\!\![-b(z_{2}(z,w,t))L^{11}(z_{2}(z,w,t),w_{2}(z,w,t))\\
&+\mu^{\prime}(z_{2}(z,w,t))L^{12}(z_{2}(z,w,t),w_{2}(z,w,t))]\mathrm{d}t.
\end{align*}

The integral equations can be resolved by using the method of successive approximations. First, define
\begin{align*}
M=
\begin{pmatrix}
  L^{11} \\
  L^{12} 
\end{pmatrix}
\doteq
\begin{pmatrix}
  M_{1} \\
  M_{2} 
\end{pmatrix}
\quad\rm{and}\quad
\varphi=
\begin{pmatrix}
  0 \\
  h_{1}(\phi^{-1}_{3}(\phi_{1}(w)+\phi_{2}(z)))
\end{pmatrix}.
\end{align*}

Consider the following operator acting on $M$:
\begin{align*}
T_{1}[M](z,w)=&\int_{0}^{t_{1}^{F}(z,w)}[-\lambda^{\prime}(z_{1}(z,w,t))L^{11}(z_{1}(z,w,t),w_{1}(z,w,t))\\
&-c(z_{1}(z,w,t))L^{12}(z_{1}(z,w,t),w_{1}(z,w,t))]\mathrm{d}t,\\
T_{2}[M](z,w)=&\int_{0}^{t_{2}^{F}(z,w)}\!\!\!\!\![-b(z_{2}(z,w,t))L^{11}(z_{2}(z,w,t),w_{2}(z,w,t))\\
&+\mu^{\prime}(z_{2}(z,w,t))L^{12}(z_{2}(z,w,t),w_{2}(z,w,t))]\mathrm{d}t.
\end{align*}
Then define
\begin{align*}
M^{0}(z,w)=\varphi(z,w),\quad\quad M^{d}(z,w)=\varphi(z,w)+\mathbf{T}[M^{d-1}](z,w) \quad(d\geq1)
\end{align*}
where
\begin{align*}
\mathbf{T}=
\begin{pmatrix}
  T_{1} \\
  T_{2}
\end{pmatrix}.
\end{align*}

Notice that if $\lim_{d\rightarrow \infty}M^{d}(z,w)$ exists, then $M=\lim_{d\rightarrow \infty}M^{d}(z,w)$ is a solution of the integral equation (and naturally, a solution of the kernel equation as well). Define increment $\Delta M^{d}=M^{d}-M^{d-1}(d\geq1)$ with $\Delta M^{0}=\varphi(z,w)$. By the linearity of operator $\mathbf{T}$, the following equation $\Delta M^{d}(z,w)=\mathbf{T}[\Delta M^{d-1}](z,w)$ holds. By using the definition of $\Delta M^{d}$, it follows that if $\sum_{d=0}^{+\infty}\Delta M^{d}(z,w)$ converges, then
\begin{align}\label{3:eq29}
M(z,w)=\sum_{d=0}^{+\infty}\Delta M^{d}(z,w).
\end{align}

Next, we will prove that the series $ \sum_{d=0}^{+\infty}\Delta M^{d}(z,w) $ converges. First, define some constants as follows
\begin{align*}
a=&\max_{-1\leq w\leq 1}\bigg\{\frac{1}{\lambda(w)},\frac{1}{\mu(w)}\bigg\},  \quad\quad \bar{h}=\max_{-1\leq w\leq 1}\{|h_{1}(w)|,|h_{2}(w)|\},\\
b_{1}=&\max_{-1\leq z\leq 1}\{|\lambda^{\prime}(z)|,|b(z)|\}, \quad\quad b_{2}=\max_{-1\leq z\leq 1}\{|\mu^{\prime}(z)|,|c(z)|\},\\
b=&b_{1}+b_{2}.
\end{align*}
Next, two lemmas need to be presented in advance.
\begin{lemma}\label{3:lem2}
For $i=1,2$, $t_{i}^{F}(z,w), w_{i}(z,w,t)$ defined in (\ref{3:eq23})-(\ref{3:eq24}) and (\ref{3:eq27})-(\ref{3:eq28}) with $d\geq1, (z,w)\in E_{1}$, it follows that
\begin{align*}
\int_{0}^{t_{i}^{F}(z,w)}w_{i}^{d}(z,w,t)\mathrm{d}t\leq a \frac{w^{d+1}}{d+1}.
\end{align*}
\end{lemma}
\begin{proof}
We only show the proof for $i=1$, which is similar for $i=2$. By equation (\ref{3:eq23}), one can write
\begin{align*}
\int_{0}^{t_{1}^{F}(z,w)}w_{1}^{d}(z,w,t)\mathrm{d}t=\int_{0}^{t_{1}^{F}(z,w)}[\phi_{1}^{-1}(t+\frac{\phi_{1}(w)-\phi_{1}(z)}{2})]^{d}\mathrm{d}t
\end{align*}
Applying a coordinate transformation $x=\phi_{1}^{-1}(t+\frac{\phi_{1}(w)-\phi_{1}(z)}{2})$ to the above integral, one can see that $x\in[\phi_{1}^{-1}(\frac{\phi_{1}(w)-\phi_{1}(z)}{2}),w]$. Then, taking the derivative of $x$ with respect to variable $t$ follows
\begin{align*}
\frac{\mathrm{d}x}{\mathrm{d}t}=\frac{1}{\phi_{1}^{\prime}(x)}=\lambda(x),
\end{align*}
and thus, the integral can be rewritten as 
\begin{align*}
\int_{0}^{t_{1}^{F}(z,w)}[\phi_{1}^{-1}(t+\frac{\phi_{1}(w)-\phi_{1}(z)}{2})]^{d}\mathrm{d}t&=\int_{\phi_{1}^{-1}(\frac{\phi_{1}(w)-\phi_{1}(z)}{2})}^{w}\frac{x^{d}}{\lambda(x)}\mathrm{d}x\\
&\leq a\int_{0}^{w}x^{d}\mathrm{d}x\\
&=a\frac{w^{d+1}}{d+1}.
\end{align*}
\end{proof}
\begin{lemma}\label{3:lem3}
For $i=1,2, d\geq1$ and $(z,w)\in E_{1},$ assume that
\begin{align*}
|\Delta M_{i}^{d}(z,w)|\leq \bar{h}\frac{a^{d}b^{d}w^{d}}{d!},
\end{align*}
and then it follows that
\begin{align*}
|\Delta M_{i}^{d+1}(z,w)|\leq \bar{h}\frac{a^{d+1}b^{d+1}w^{d+1}}{(d+1)!}.
\end{align*}
\end{lemma}
\begin{proof}
For $i=1$, by using Lemma \ref{3:lem2}, it can be acquired
\begin{align*}
|\Delta M_{1}^{d+1}(z,w)|&=|T_{1}[\Delta M^{d}](z,w)|\\
&\leq\int_{0}^{t_{1}^{F}(z,w)}[b_{1}|\Delta M_{1}^{d}(z_{1}(z,w,t),w_{1}(z,w,t))|\\
&\quad\quad\quad\quad\quad+b_{2}|\Delta M_{2}^{d}(z_{1}(z,w,t),w_{1}(z,w,t))|]\mathrm{d}t\\
&\leq\int_{0}^{t_{1}^{F}(z,w)}\bar{h}\frac{a^{d}b^{d+1}}{d!}w_{1}^{d}(z,w,t)\mathrm{d}t\\
&\leq \bar{h}\frac{a^{d+1}b^{d+1}w^{d+1}}{(d+1)!}.
\end{align*}

The detailed proof for case $i=2$ is similar to the one of $i=1$, which is omitted here consequently.
\end{proof}

$\forall i=1,2$, it is easy to find $|\Delta M_{i}^{0}(z,w)|\leq \bar{h}=\bar{h}\frac{a^{0}b^{0}w^{0}}{(0)!}$. By Lemma \ref{3:lem3} and induction, one can have
\begin{align*}
|\Delta M_{i}^{d}(z,w)|\leq \bar{h}\frac{a^{d}b^{d}w^{d}}{d!},\quad d=0,1,2,\cdot\cdot\cdot.
\end{align*}
And thus
\begin{align}\label{3:eq30}
\bigg|\sum_{d=0}^{+\infty}\Delta M_{i}^{d}(z,w)\bigg|\leq \bar{h}e^{abw}.
\end{align}
Now, it can be concluded that the series $ \sum_{d=0}^{+\infty}\Delta M^{d}(z,w) $ is bounded and converges uniformly. Thus, a bounded solution of (\ref{3:eq14})-(\ref{3:eq16}) exists. So far, the existence part of Theorem \ref{3:them1}.

The statement of the uniqueness of the kernel is easy, indeed, if there are two different solutions $\bar{M}(z,w)$ and $\tilde{M}(z,w)$ of equations (\ref{3:eq14})-(\ref{3:eq16}). Define $M(z,w)=\bar{M}(z,w)-\tilde{M}(z,w)$. By linearity of (\ref{3:eq14})-(\ref{3:eq16}),  $M(z,w)$ also satisfies (\ref{3:eq14})-(\ref{3:eq16}), with $h_{1}(w)=0$ ($h_{2}(w)=0$ is similar). Then $\bar{h}=0$, and then it can be concluded $M(z,w)=0$ from (\ref{3:eq30}), which implies that $\bar{M}(z,w)=\tilde{M}(z,w)$.

To demonstrate the continuity of the solution, it is essential to observe the continuity of each individual term, i.e., $\forall d=0,1,2,\dots, \Delta M^{d}(z,w)\in C(E_{1})$, which can be established on account of the uniform convergence of equation (\ref{3:eq29}). First, for $d=0$, it is easy to see that $\Delta M^{0}=\varphi(z,w)\in C(E_{1})$. Assume that $\Delta M^{d}\in C(E_{1}), d\geq1$, and then $\Delta M^{d+1}=\mathbf{T}[\Delta M^{d}]\in C(E_{1})$. Indeed, by the definition of $\mathbf{T}$ (the integral of a continuous function), it is known that $\Delta M^{d+1}\in C(E_{1})$ is continuous. Thus, by induction, the continuity of the solution has been proven.

Finally, the proof of Theorem \ref{3:them1} is completed.
\end{proof}

Now it is ready to state the stabilization result of Section \ref{3:subsec3.1} as follows.
\begin{theorem}\label{3:them2}
Consider system (\ref{3:eq1})-(\ref{3:eq3}) with the following feedback control laws
\begin{align}
  U_{1}(t) & =-\int_{-1}^{1}u(z,t)L^{11}(z,-1)\mathrm{d}z-\int_{-1}^{1}v(z,t)L^{12}(z,-1)\mathrm{d}z, \label{3:eq31}\\
  U_{2}(t) & =\int_{-1}^{1}u(z,t)L^{21}(z,1)\mathrm{d}z+\int_{-1}^{1}v(z,t)L^{22}(z,1)\mathrm{d}z. \label{3:eq32}
\end{align}
For any $(u_{0},v_{0})\in L^{2}([0,1])$, the zero equilibrium state is reached in finite time.
\end{theorem}
\begin{proof}
To obtain control laws (\ref{3:eq31}) and (\ref{3:eq32}), substitute $ w = -1 $ and $ w = 1 $ into the transformations (\ref{3:eq12}) and (\ref{3:eq13}), respectively. According to Remark \ref{3:rem4}, we know that transformations (\ref{3:eq12}) -(\ref{3:eq13}) is invertible, when applying control law (\ref{3:eq31})-(\ref{3:eq32}) the dynamical behavior of original system (\ref{3:eq1})-(\ref{3:eq3}) is the same as the behavior of target system (\ref{3:eq7})-(\ref{3:eq9}). Lemma \ref{3:lem1} implies that $(\alpha,\beta)$ goes to zero in finite time. Therefore, $(u,v)$ also converge to the zero equilibrium state in finite time.
\end{proof}

\subsection{Case 2: $\lambda(w)>\mu(-w)$}\label{3:subsec3.2}
In this case, a new target system will be used, which is different from Case 1
\begin{numcases}{}
\frac{\partial \alpha}{\partial t}(w,t)+\lambda(w)\frac{\partial \alpha}{\partial w}(w,t)=0\label{3:eq33}\\
\frac{\partial \beta}{\partial t}(w,t)-\mu(w)\frac{\partial \beta}{\partial w}(w,t)=p(w)\alpha(-w,t)+\int_{-w}^{w}D^{+}(z,w)\alpha(z,t)\mathrm{d}z\nonumber\\
\quad\quad\quad\quad\quad\quad\quad\quad\quad\quad\quad\quad\quad\quad\quad\quad\quad+\int_{-w}^{w}D^{-}(z,w)\beta(z,t)\mathrm{d}z\label{3:eq34}\\
\alpha(-1,t)=\beta(1,t)=0,\label{3:eq35}
\end{numcases}
where $p(w)\in C([-1,1])$; $D^{+}(z,w)$ and $D^{-}(z,w)\in L^{\infty}(E)$ will be determined later; initial conditions $\alpha(w,0)=\alpha_{0}(w)$ and $\beta(w,0)=\beta_{0}(w)$ belong to $L^{2}([-1,1])$. 
\begin{remark}\label{3:rem5}
The purpose of using a target system different from that in Case 1 is to ensure the well-posedness of the kernel equation. The integral coupling appearing in equation (\ref{3:eq34}) are solely used for control design. Since the integral terms are feedforward terms, it does not affect the stability of the target system.
\end{remark}
\begin{lemma}\label{3:lem4}
System (\ref{3:eq33})-(\ref{3:eq35}) reaches its zero equilibrium state in finite time.
\end{lemma}
\begin{proof}
Similar to Lemma \ref{3:lem1}, first, when $t\geq \phi_{1}(1)-\phi_{1}(-1)$, $\alpha\equiv0$. Thus, equation (\ref{3:eq34}) is written as
\begin{align*}
\frac{\partial \beta}{\partial t}(w,t)-\mu(w)\frac{\partial \beta}{\partial w}(w,t)=\int_{-w}^{w}D^{-}(z,w)\beta(z,t)\mathrm{d}z.
\end{align*}
Similar to Lemma \ref{3:lem1},  when $t\geq \phi_{1}(1)-\phi_{1}(-1)+\phi_{2}(1)-\phi_{2}(-1)=\phi_{3}(1)-\phi_{3}(-1)$, $\beta\equiv0$.
\end{proof}

Using the same backstepping transformation (\ref{3:eq12})-(\ref{3:eq13}) to map original system(\ref{3:eq1})-(\ref{3:eq3}) to target system (\ref{3:eq33})-(\ref{3:eq35}), and similarly, two uncoupled kernel PDEs can be obtained as follows.

First, kernels $L^{11}$ and $L^{12}$ safisfy 
\begin{align}
\lambda(w)L_{w}^{11}(z,w)+\lambda(z)L_{z}^{11}(z,w)=-\lambda^{\prime}(z)L^{11}(z,w)-c(z)L^{12}(z,w), \label{3:eq36}\\
\lambda(w)L_{w}^{12}(z,w)-\mu(z)L_{z}^{12}(z,w)=-b(z)L^{11}(z,w)+\mu^{\prime}(z)L^{12}(z,w), \label{3:eq37}
\end{align}
with boundary conditions
\begin{align}
L^{11}(-w,w)=0,\quad L^{12}(w,w)=\frac{b(w)}{\lambda(w)+\mu(w)}\doteq h_{1}(w), \quad L^{12}(-w,w)=0. \label{3:eq38}
\end{align}
\begin{assumption}\label{3:assum2}
$h_{1}$ satisfies the $C^{0}$ compatibility condition at point $(z, w ) = (0, 0)$, i.e., $h_{1}(0)=0$.
\end{assumption}

Second, kernels $L^{21}$ and $L^{22}$ safisfy
\begin{align}
\mu(w)L_{w}^{22}(z,w)+\mu(z)L_{z}^{22}(z,w)=b(z)L^{21}(z,w)-\mu^{\prime}(z)L^{22}(z,w), \label{3:eq39}\\
\mu(w)L_{w}^{21}(z,w)-\lambda(z)L_{z}^{21}(z,w)=\lambda^{\prime}(z)L^{21}(z,w)+c(z)L^{22}(z,w), \label{3:eq40}
\end{align}
with boundary conditions
\begin{align}
L^{22}(-w,w)=0,\quad\quad L^{21}(w,w)=-\frac{c(w)}{\lambda(w)+\mu(w)}\doteq h_{2}(w). \label{3:eq41}
\end{align}
Both equations (\ref{3:eq36})-(\ref{3:eq38}) and (\ref{3:eq39})-(\ref{3:eq41}) evolve separately in domain $E=\{(z,w)\in \mathbb{R}^{2}|-1\leq-w\leq z\leq w\leq 1\}$.

Besides, $p(w)$ is given by
\begin{align*}
p(w)=(\lambda(-w)-\mu(w))L^{21}(-w,w).
\end{align*}
Furthermore, the following equations are obtained for $D^{+}$ and $D^{-}$:\\
for $z\in[0,w]$
\begin{align*}
-p(w)L^{11}(z,-w)=&D^{+}(z,w)-\int_{z}^{w}D^{+}(s,w)L^{11}(z,s)+D^{-}(s,w)L^{21}(z,s)\mathrm{d}s\\
&-\int_{-z}^{-w}D^{+}(s,w)L^{11}(z,s)+D^{-}(s,w)L^{21}(z,s)\mathrm{d}s, \\
-p(w)L^{12}(z,-w)=&D^{-}(z,w)-\int_{z}^{w}D^{+}(s,w)L^{12}(z,s)+D^{-}(s,w)L^{22}(z,s)\mathrm{d}s\\
&-\int_{-z}^{-w}D^{+}(s,w)L^{12}(z,s)+D^{-}(s,w)L^{22}(z,s)\mathrm{d}s,
\end{align*}
for $z\in[-w,0]$
\begin{align*}
-p(w)L^{11}(z,-w)=&D^{+}(z,w)-\int_{-z}^{w}D^{+}(s,w)L^{11}(z,s)+D^{-}(s,w)L^{21}(z,s)\mathrm{d}s\\
&-\int_{z}^{-w}D^{+}(s,w)L^{11}(z,s)+D^{-}(s,w)L^{21}(z,s)\mathrm{d}s, \\
-p(w)L^{12}(z,-w)=&D^{-}(z,w)-\int_{-z}^{w}D^{+}(s,w)L^{12}(z,s)+D^{-}(s,w)L^{22}(z,s)\mathrm{d}s\\
&-\int_{z}^{-w}D^{+}(s,w)L^{12}(z,s)+D^{-}(s,w)L^{22}(z,s)\mathrm{d}s.
\end{align*}
Through a simple coordinate transformation, for $z\in[0,w]$, one can get
\begin{align}
-p(w)L^{11}(z,-w)=&D^{+}(z,w)-\int_{z}^{w}[D^{+}(s,w)L^{11}(z,s)+D^{-}(s,w)L^{21}(z,s)\nonumber\\
&\quad\quad\quad-D^{+}(-s,w)L^{11}(z,-s)-D^{-}(-s,w)L^{21}(z,-s)]\mathrm{d}s, \label{3:eq42}\\
-p(w)L^{12}(z,-w)=&D^{-}(z,w)-\int_{z}^{w}[D^{+}(s,w)L^{12}(z,s)+D^{-}(s,w)L^{22}(z,s)\nonumber\\
&\quad\quad\quad-D^{+}(-s,w)L^{12}(z,-s)-D^{-}(-s,w)L^{22}(z,-s)]\mathrm{d}s,\label{3:eq43}\\
-p(w)L^{11}(-z,-w)=&D^{+}(-z,w)-\int_{z}^{w}[D^{+}(s,w)L^{11}(-z,s)+D^{-}(s,w)L^{21}(-z,s)\nonumber\\
&\quad-D^{+}(-s,w)L^{11}(-z,-s)-D^{-}(-s,w)L^{21}(-z,-s)]\mathrm{d}s, \label{3:eq44}\\
-p(w)L^{12}(-z,-w)=&D^{-}(-z,w)-\int_{z}^{w}[D^{+}(s,w)L^{12}(-z,s)+D^{-}(s,w)L^{22}(-z,s)\nonumber\\
&\quad-D^{+}(-s,w)L^{12}(-z,-s)-D^{-}(-s,w)L^{22}(-z,-s)]\mathrm{d}s.\label{3:eq45}
\end{align}
\begin{remark}\label{3:rem6}
For every $w\in[-1,1]$, define
\begin{align*}
\begin{cases}
  D^{+}(z,w)=D_{1}(z,w),  \\
  D^{-}(z,w)=D_{2}(z,w),  \\
  D^{+}(-z,w)=D_{3}(z,w), \\
  D^{-}(-z,w)=D_{4}(z,w),
\end{cases}
\end{align*}
where $z\in[0,w]$. It is evident that equations (\ref{3:eq42})-(\ref{3:eq45}) represent classical Volterra integral equations of second kind on interval $[0,w]$, with function $D_{i}(\cdot,w), i=1,2,3,4$ being the variable under consideration. Consequently, provided kernels $L^{11},L^{12},L^{21},L^{22}$ and $p$ are well-defined and bounded, so are $D_{i}, i=1,2,3,4$.
\end{remark}

Next, the well-posedness theorem will provided for kernel $L$.
\begin{theorem}\label{3:them3}
Hyperbolic PDEs (\ref{3:eq36})-(\ref{3:eq41}) exists a unique solution $L^{11},L^{21},L^{22} \in C(E), L^{12}\in L^{\infty}(E).$
\end{theorem}
\begin{proof}
Notice that kernel equations (\ref{3:eq39})-(\ref{3:eq41}) share the same structure as (\ref{3:eq17})-(\ref{3:eq19}). Using the method similar to Theorem \ref{3:them1}, the well-posedness of $L^{21}$ and $L^{22}$ is easily obtained. Therefore, we will only provide the proof of well-posedness for kernels $L^{11}$ and $L^{12}$.

Similarly to Theorem \ref{3:them1}, by utilizing the symmetry of $E_{1}=\{(z,w)\in \mathbb{R}^{2}|-w\leq z\leq w, 0\leq w\leq 1\}$ and $E_{2}=\{(z,w)\in \mathbb{R}^{2}|w\leq z\leq -w, -1\leq w\leq 0\}$, the well-posedness only need to be proven in region $E_{1}$.

First, define
\begin{align}\label{3:eq46}
\phi_{4}(w)=\phi_{1}(w)+\phi_{2}(-w)
\end{align}
where $\phi_{1}$ and $\phi_{2}$ are defined by (\ref{3:eq10}) and (\ref{3:eq11}). Note that $\phi_{4}$ is monotonically decreasing, invertible and continuous differentiable.

Characteristics of kernel $L^{11}$: Since equations (\ref{3:eq36}) and (\ref{3:eq14}) have exactly the same structure and the same boundary condition, the characteristic lines defined in (\ref{3:eq21})-(\ref{3:eq24}) will still be used.
\begin{figure}
  \centering
  \includegraphics[width=7cm,height=5.5cm]{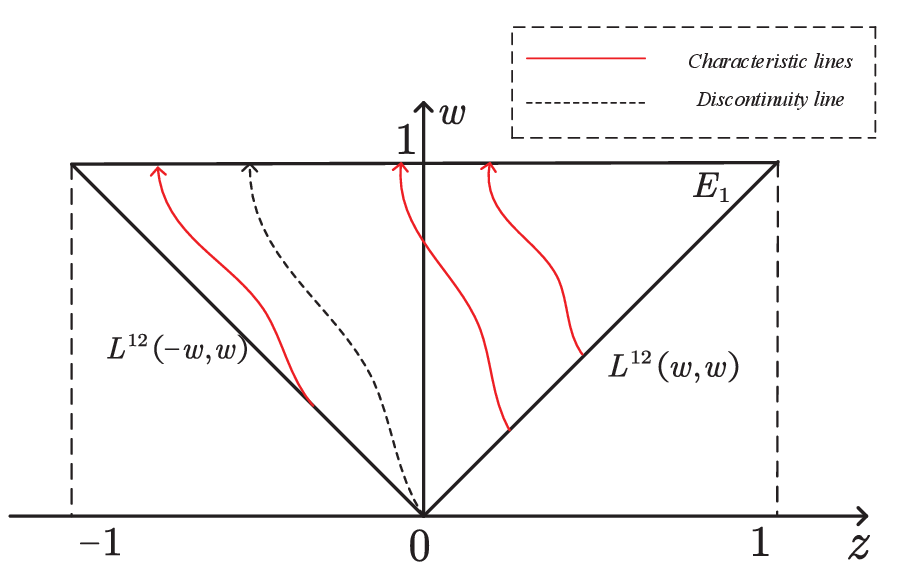}
  \caption{Characteristics lines for kernel $L^{12}$}\label{3:fig2}
\end{figure}

Characteristics of the $L^{12}$ kernel: $\forall (z,w)\in E_{1}$, define the following characteristic lines $(z_{3}(z,w,\cdot),w_{3}(z,w,\cdot))$
\begin{align}
&\begin{cases}\label{3:eq47}
\frac{\mathrm{d}w_{3}}{\mathrm{d}t}(z,w,t)=\lambda(w_{3}(z,w,t)),\quad t\in[0,t_{3}^{F}(z,w)],\\
w_{3}(z,w,0)=w_{3}^{0}(z,w),\quad w_{3}(z,w,t_{3}^{F}(z,w))=w,
\end{cases}\\
&\begin{cases}\label{3:eq48}
\frac{\mathrm{d}z_{3}}{\mathrm{d}t}(z,w,t)=-\mu(z_{3}(z,w,t)),\quad t\in[0,t_{3}^{F}(z,w)],\\
z_{3}(z,w,0)=z_{3}^{0}(z,w)=\epsilon(z,w)w_{3}^{0}(z,w),\quad z_{3}(z,w,t_{3}^{F}(z,w))=z,
\end{cases}
\end{align}
where
\begin{align*}
\epsilon(z,w)=
\begin{cases}
1, \quad\quad (z,w)\in \mathcal{T}_{1},\\
-1, \quad~ (z,w)\in \mathcal{T}_{2},
\end{cases}
\end{align*}
with
\begin{align*}
\begin{cases}
\mathcal{T}_{1}=\{(z,w)\in E_{1}|\phi_{1}(w)+\phi_{2}(z)\geq 0\},\\
\mathcal{T}_{2}=\{(z,w)\in E_{1}|\phi_{1}(w)+\phi_{2}(z)< 0\}.
\end{cases}
\end{align*}

\begin{remark}\label{3:rem7}
It can be seen through direct calculation that the equation for the discontinuity line in Fig. \ref{3:fig2} is $\phi_{1}(w)+\phi_{2}(z)=0$.
\end{remark}

These lines (refer to Fig. \ref{3:fig2}) start at coordinate $(z_{3}^{0}(z,w),w_{3}^{0}(z,w))$ and end at position $(z,w)$. Further, we also can calculate that
\begin{align}\label{3:eq49}
w_{3}(z,w,t)=
\begin{cases}
\phi_{1}^{-1}(t+\phi_{1}(\phi_{3}^{-1}(\phi_{1}(w)+\phi_{2}(z)))),\quad  (z,w)\in \mathcal{T}_{1}\\
\phi_{1}^{-1}(t+\phi_{1}(\phi_{4}^{-1}(\phi_{1}(w)+\phi_{2}(z)))),\quad  (z,w)\in \mathcal{T}_{2}
\end{cases}
\end{align}
\begin{align}\label{3:eq50}
z_{3}(z,w,t)=
\begin{cases}
\phi_{2}^{-1}(-t+\phi_{2}(-\phi_{3}^{-1}(\phi_{1}(w)+\phi_{2}(z)))),\quad  (z,w)\in \mathcal{T}_{1}\\
\phi_{2}^{-1}(-t+\phi_{2}(-\phi_{4}^{-1}(\phi_{1}(w)+\phi_{2}(z)))),\quad  (z,w)\in \mathcal{T}_{2}
\end{cases}
\end{align}
where $t\in[0,t_{3}^{F}(z,w)]$ with
\begin{align}\label{3:eq51}
t_{3}^{F}(z,w)=
\begin{cases}
\phi_{1}(w)-\phi_{1}(\phi_{3}^{-1}(\phi_{1}(w)+\phi_{2}(z))),\quad(z,w)\in \mathcal{T}_{1}\\
\phi_{1}(w)-\phi_{1}(\phi_{4}^{-1}(\phi_{1}(w)+\phi_{2}(z))),\quad(z,w)\in \mathcal{T}_{2}
\end{cases}
\end{align}

By integrating kernel equations (\ref{3:eq36}) and (\ref{3:eq37}) along with their respective characteristic lines and utilizing boundary conditions (\ref{3:eq38}), it is obtained 
\begin{align*}
L^{11}(z,w)=\int_{0}^{t_{1}^{F}(z,w)}&[-\lambda^{\prime}(z_{1}(z,w,t))L^{11}(z_{1}(z,w,t),w_{1}(z,w,t))\\
&-c(z_{1}(z,w,t))L^{12}(z_{1}(z,w,t),w_{1}(z,w,t))]\mathrm{d}t
\end{align*}
and
\begin{align*}
L^{12}(z,w)=&(1-\delta(z,w))h_{1}(\phi^{-1}_{3}(\phi_{1}(w)+\phi_{2}(z)))\\
&+\int_{0}^{t_{3}^{F}(z,w)}\!\!\!\!\![-b(z_{3}(z,w,t))L^{11}(z_{3}(z,w,t),w_{3}(z,w,t))\\
&\quad\quad\quad+\mu^{\prime}(z_{3}(z,w,t))L^{12}(z_{3}(z,w,t),w_{3}(z,w,t))]\mathrm{d}t,
\end{align*}
where
\begin{align}
\delta(z,w)=
\begin{cases}
0,\quad(z,w)\in \mathcal{T}_{1}\\
1,\quad(z,w)\in \mathcal{T}_{2}.
\end{cases}
\end{align}

With the similar method (successive approximations) of Theorem \ref{3:them1}, one can get the well-posedness of kernel $L$, so the specifics are omitted here.
\end{proof}

\begin{remark}\label{3:rem8}
Kernel $L^{12}$ cannot be continuous across the entire region $E$ due to the emergence of discontinuous lines (see Fig. \ref{3:fig2}). Indeed, kernel $L^{12}$ only loses its continuity on the discontinuous lines while remaining continuous elsewhere.
\end{remark}

We are now ready to state the stabilization result of Section \ref{3:subsec3.2} as follows.
\begin{theorem}\label{3:them4}
Consider system (\ref{3:eq1})-(\ref{3:eq3}) with the following feedback control laws
\begin{align}
  U_{1}(t) & =-\int_{-1}^{1}u(z,t)L^{11}(z,-1)\mathrm{d}z-\int_{-1}^{1}v(z,t)L^{12}(z,-1)\mathrm{d}z, \label{3:eq53}\\
  U_{2}(t) & =\int_{-1}^{1}u(z,t)L^{21}(z,1)\mathrm{d}z+\int_{-1}^{1}v(z,t)L^{22}(z,1)\mathrm{d}z. \label{3:eq54}
\end{align}
For any $(u_{0},v_{0})\in L^{2}([0,1])$, the zero equilibrium state can be reached in finite time.
\end{theorem}
\begin{proof}
Similar to the proof of Theorem \ref{3:them2}, the details are omitted here.
\end{proof}

\subsection{Case 3: $\lambda(w)<\mu(-w)$}\label{3:subsec3.3}
The discussion in this section is entirely similar to that in Section \ref{3:subsec3.2}, with only the corresponding adjustment to the target system:
\begin{numcases}{}
\frac{\partial \alpha}{\partial t}(w,t)+\lambda(w)\frac{\partial \alpha}{\partial w}(w,t)=q(w)\beta(-w,t)+\int_{-w}^{w}K^{+}(z,w)\alpha(z,t)\mathrm{d}z\nonumber\\
\quad\quad\quad\quad\quad\quad\quad\quad\quad\quad\quad\quad\quad\quad\quad\quad\quad+\int_{-w}^{w}K^{-}(z,w)\beta(z,t)\mathrm{d}z\label{3:eq55}\\
\frac{\partial \beta}{\partial t}(w,t)-\mu(w)\frac{\partial \beta}{\partial w}(w,t)=0\label{3:eq56}\\
\alpha(-1,t)=\beta(1,t)=0,\label{3:eq57}
\end{numcases}
where $q(w)\in C([-1,1])$; $K^{+}(z,w)$ and $K^{-}(z,w)\in L^{\infty}(E)$ will be determined later; initial conditions $\alpha(w,0)=\alpha_{0}(w)$ and $\beta(w,0)=\beta_{0}(w)$ belong to $L^{2}([-1,1])$. 
\begin{remark}\label{3:rem9}
Similar to Lemma \ref{3:lem4}, target system (\ref{3:eq55})-(\ref{3:eq57}) converges to the zero equilibrium state in finite time can also be concluded.
\end{remark}

Using the same backstepping transformation (\ref{3:eq12})-(\ref{3:eq13}) to obtain kernel equations similar to (\ref{3:eq36})-(\ref{3:eq41}), and can further derive
\begin{align*}
q(w)=(\lambda(w)-\mu(-w))L^{12}(-w,w).
\end{align*}

Furthermore, the following equations are obtained for $K^{+}$ and $K^{-}$:\\
for $z\in[0,w]$
\begin{align*}
-q(w)L^{21}(z,-w)=&K^{+}(z,w)-\int_{z}^{w}K^{+}(s,w)L^{11}(z,s)+K^{-}(s,w)L^{21}(z,s)\mathrm{d}s\\
&-\int_{-z}^{-w}K^{+}(s,w)L^{11}(z,s)+K^{-}(s,w)L^{21}(z,s)\mathrm{d}s, \\
-q(w)L^{22}(z,-w)=&K^{-}(z,w)-\int_{z}^{w}K^{+}(s,w)L^{12}(z,s)+K^{-}(s,w)L^{22}(z,s)\mathrm{d}s\\
&-\int_{-z}^{-w}K^{+}(s,w)L^{12}(z,s)+K^{-}(s,w)L^{22}(z,s)\mathrm{d}s,
\end{align*}
for $z\in[-w,0]$
\begin{align*}
-q(w)L^{21}(z,-w)=&K^{+}(z,w)-\int_{-z}^{w}K^{+}(s,w)L^{11}(z,s)+K^{-}(s,w)L^{21}(z,s)\mathrm{d}s\\
&-\int_{z}^{-w}K^{+}(s,w)L^{11}(z,s)+K^{-}(s,w)L^{21}(z,s)\mathrm{d}s, \\
-q(w)L^{22}(z,-w)=&K^{-}(z,w)-\int_{-z}^{w}K^{+}(s,w)L^{12}(z,s)+K^{-}(s,w)L^{22}(z,s)\mathrm{d}s\\
&-\int_{z}^{-w}K^{+}(s,w)L^{12}(z,s)+K^{-}(s,w)L^{22}(z,s)\mathrm{d}s.
\end{align*}
The uniqueness and boundedness of $K^{+}$ and $K^{-}$ can be obtained similarly to Remark \ref{3:rem6}.

The well-posedness of the kernel $L$ is similar to Theorem \ref{3:them3}, and will not be repeated here. However, for the sake of completeness of this paper, the control laws will be directly provided here. 
\begin{theorem}\label{3:them5}
Consider system (\ref{3:eq1})-(\ref{3:eq3}) with the following feedback control law:
\begin{align}
  U_{1}(t) & =-\int_{-1}^{1}u(z,t)L^{11}(z,-1)\mathrm{d}z-\int_{-1}^{1}v(z,t)L^{12}(z,-1)\mathrm{d}z, \label{3:eq58}\\
  U_{2}(t) & =\int_{-1}^{1}u(z,t)L^{21}(z,1)\mathrm{d}z+\int_{-1}^{1}v(z,t)L^{22}(z,1)\mathrm{d}z. \label{3:eq59}
\end{align}
For any $(u_{0},v_{0})\in L^{2}([0,1])$, the zero equilibrium state can be reached in finite time.
\end{theorem}
\begin{proof}
Similar to the proof of Theorem \ref{3:them2}, the details are still omitted.
\end{proof}
\begin{remark}\label{3:rem10}
Superficially, the controllers in Theorem \ref{3:them2}, \ref{3:them4} and \ref{3:them5} appear identical, but fundamentally, they differ widely according to the distinct kernel of each controller. It can be seen that Case 1 has the optimal control effect, with a shorter convergence time to zero equilibrium state than that of Case 2 and 3.
\end{remark}

\section{Numerical Simulation}\label{3:sec4}
In this section, a simulation example will be conducted to illustrate the correctness of our theory. The numerical example is illustrated as follows
\begin{align}\label{3:eq60}
\begin{cases}
u_{t}+(3+w^{2})u_{w}=(3\mathrm{e}^{3w})v,\\
v_{t}-(2+w^{4})v_{w}=(1+w)u,\\
\end{cases}
\end{align}
where the initial states are chosen as $u_{0}(w)=w^{2}, v_{0}(w)=\mathrm{e}^{w}$.

Controllers (\ref{3:eq53})-(\ref{3:eq54}) are employed to stabilize system (\ref{3:eq60}), where the control kernels use the numerical solution derived from the algorithm in Reference \cite{bib3.35}. 
From Fig. \ref{3:fig3}, it can be seen that the open-loop system (\ref{3:eq60}) is unstable, but the closed-loop system converges in finite time. Figs. \ref{3:fig4} and \ref{3:fig6} respectively show the variations of the states and control signal with respect to time $t$. The aforementioned results indicate the feasibility of our control scheme.

\begin{figure}
  \centering
  \includegraphics[width=7cm,height=5.5cm]{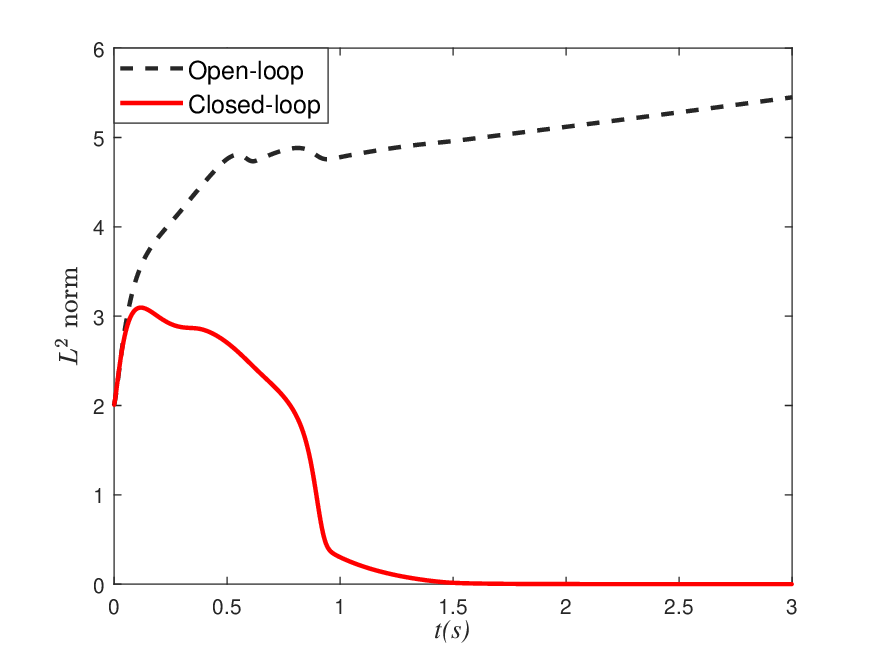}
  \caption{$L^{2}$-norm of open-loop and closed-loop system states}\label{3:fig3}
\end{figure}

\begin{figure}[htbp]
\begin{minipage}[t]{0.5\textwidth}
\centering
\includegraphics[width=\textwidth]{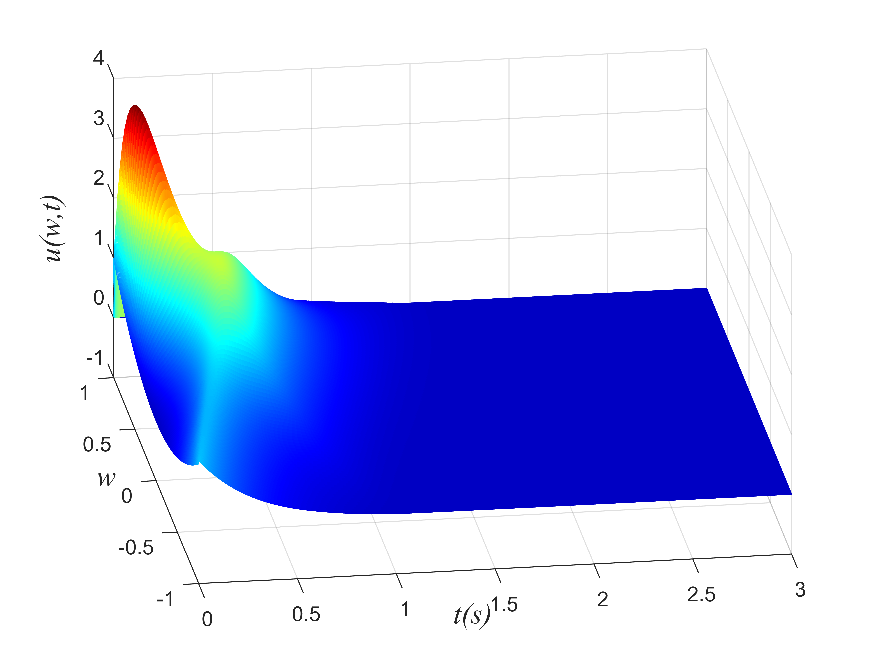}
\caption{State $u$ and $v$}
\end{minipage}
\begin{minipage}[t]{0.5\textwidth}
\centering
\includegraphics[width=\textwidth]{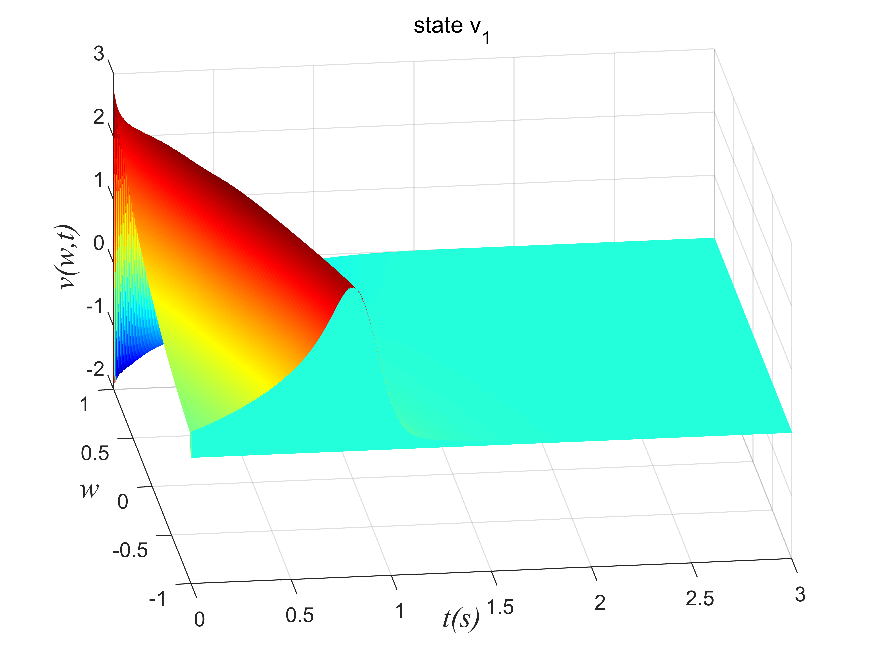}
\end{minipage}
\end{figure}\label{3:fig4}

\begin{figure}
  \centering
  \includegraphics[width=7cm,height=5.5cm]{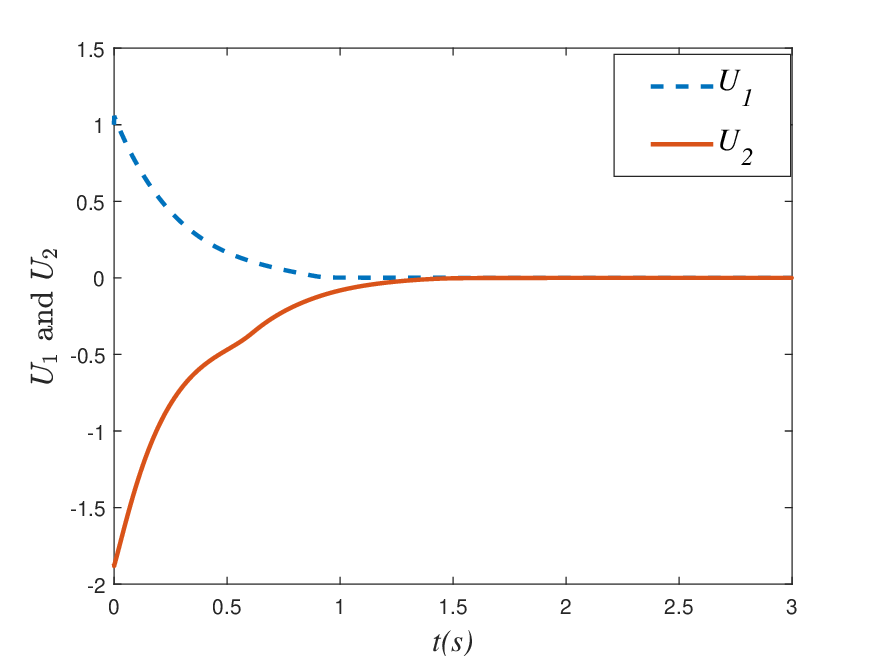}
  \caption{Control signal of $U_{1}$ and $U_{2}$}\label{3:fig6}
\end{figure}

\section{Conclusion}\label{3:sec5}

This paper extends the infinite-dimensional backstepping method for 1-D linear hyperbolic system with two actuators. Using the backstepping method, full-state feedback controllers are respectively designed for three scenarios of transport coefficients, ensuring that the closed-loop system state converges to zero in finite time. Our results provide the potential for more complex designs such as bilateral sensor/actuator output feedback and fault-tolerant designs. 

Future research will focus on achieving bilateral controller design for general $(n+m)\times(n+m)$ heterodirectional coupled hyperbolic systems with spatially varying coefficients. This will make the situation more complex and challenging, but we are confident that it is achievable.

\begin{itemize}
\item Funding

This research is suppported in part by Natural Science Basic Research Program of Shaanxi under Grant no. 2024JC-YBMS538; Xidian University Specially Funded Project for Interdisciplinary Exploration, Grant no. TZJH2024003.

\end{itemize}

\bibliography{sn-bibliography}

\end{document}